\documentclass[10pt,a4paper,reqno]{amsart}

\usepackage[UKenglish]{babel}
\usepackage{latexsym, amssymb, amsfonts, enumerate}

\usepackage{amsmath} 
\usepackage{amsthm}
\usepackage{amsopn}
\usepackage{mathrsfs}
\DeclareMathSymbol{\R}{\mathalpha}{AMSb}{"52}

\newcommand{\N}{\mathbb{N}}

\newcommand{\F}{\mathcal{F}}
\newcommand{\E}{\mathbb{E}}
\renewcommand{\P}{\mathbb{P}}

\usepackage{dsfont}
\newcommand{\one}[1]{\mathds{#1}}

\newcommand{\weaklyto}{\rightharpoonup}

\theoremstyle{plain}
\newtheorem{anytheorem}{Theorem}[section] 
\newtheorem{theorem}[anytheorem]{Theorem}
\newtheorem{lemma}[anytheorem]{Lemma}
\newtheorem{corollary}[anytheorem]{Corollary}
\theoremstyle{definition}
\newtheorem{definition}[anytheorem]{Definition} 
 
\newtheorem{remark}[anytheorem]{Remark} 
\newtheorem{assumption}{Assumption}

\numberwithin{equation}{section}

\begin{document}

\author{Istv\'an Gy\"ongy, Sotirios Sabanis, David \v{S}i\v{s}ka}

\title[Tamed Euler schemes for stochastic evolution equations]{Convergence of tamed Euler schemes for a class of stochastic evolution equations}

\date{\today}

\maketitle

\begin{abstract}
We prove stability and convergence of a full discretization
for a class of stochastic evolution equations
with super-linearly growing operators appearing in the drift term. 
This is done by using the recently developed tamed Euler method, 
which employs a fully explicit time stepping, coupled with a Galerkin scheme for the spatial discretization.

\end{abstract}

\section{Introduction}
In this paper we investigate the convergence of full discretizations, explicit in time, of stochastic evolution equations 
\begin{equation}
\label{eq:see}
du(t) = A u(t) dt + B u(t) dW(t)\, , \,\, t \in [0,T]
\end{equation}
with the drift term
governed by a super-linearly growing operator.
When the operator appearing in the drift term grows at most linearly
then the classical explicit Euler scheme applied to stochastic evolution equations is convergent
(when coupled appropriately with the spatial discretization),
see, for example, Gy\"ongy and Millet~\cite{gyongy:millet:on:discretization}.
If the operator appearing in the drift term grows faster than linearly then one would, in general, not expect the explicit Euler scheme to be convergent
(this is the case even in the setting of fully deterministic evolution equations).
Instead, one would typically consider the implicit Euler scheme which is convergent in this situation (see, for example, Gy\"ongy and Millet~\cite{gyongy:millet:on:discretization}).
The price one pays is the increased computational effort required at each time step of the numerical scheme.

Hutzenthaler, Jentzen and Kloeden~\cite{HJK-div} have observed that 
the absolute moments of explicit Euler approximations for 
stochastic differential equations with super-linearly growing coefficients 
may diverge to infinity at finite time. 
This led to development of ``tamed'' Euler schemes for stochastic 
differential equations. 
This was pioneered in Hutzenthaler, Jentzen and Kloeden~\cite{hutzenthaler:jentzen:kloeden:strong:convergence}
and, using different techniques, in Sabanis~\cite{sabanis:note}.
A taming-like technique in the form of truncation has been proposed
by Roberts and Tweedie~\cite{roberts:tweedie:exponential} in the
context of Markov chain Monte Carlo methods.
Further work on explicit numerical approximations of stochastic differential 
equations with super-linearly growing coefficients can be found
in Tretyakov and Zhang~\cite{tretyakov:zhang:fundamental}, 
Hutzenthaler and Jentzen~\cite{hutzenthaler:jentzen:numerical},
Sabanis~\cite{sabanis:euler} 
as well  Dareiotis, Kumar and Sabanis~\cite{dareitos:kumar:sabanis}.

Moreover Hutzenthaler, Jentzen and Kloeden~\cite{hutzenthaler:jentzen:kloeden:strong:divergence:multilevel} have demonstrated
that to apply multilevel Monte Carlo methods (see Heinrich~\cite{heinrich:monte}, \cite{heinrich:multilevel} and Giles~\cite{giles:multilevel}) 
to stochastic differential equations with super-linearly growing coefficients 
one must ``tame'' the explicit Euler scheme.
In this paper we use the idea of ``taming'' to devise a new type of a convergent explicit scheme for a class of stochastic evolution equations with super-linearly growing operators in the drift term.

The article is organised as follows. In Section~\ref{sec:main} we 
introduce the numerical scheme, give the precise assumptions and state 
the main result in Theorem~\ref{thm:main}.
Essential a priori estimates are proved in Section~\ref{sec:apriori_est}.
In Section~\ref{sec:convergence} we first use the a priori estimates 
and a compactness argument to extract weakly convergent subsequences 
and limits of the approximation.
The remaining step is to identify the weak limit of the approximation
of the nonlinear term with the nonlinear term in the equation.
This is done using a monotonicity argument in Section~\ref{sec:convergence}
where Theorem~\ref{thm:main} is finally proved.
In Section~\ref{sec:examples} we provide examples of stochastic partial
differential equations where the numerical scheme can be applied.

%

\section{Main results}
\label{sec:main}
Let $T>0$.
Let $(\Omega, \F, \P)$ be a probability space and let $(\F_t)_{t\in [0,T]}$ 
be a filtration such that $\F_0$ contains all 
the $\P$-null sets of $\F$.

Let $K>0$ and $p \in [2,\infty)$ be given constants.
Let $p^* := p / (p-1)$.
For a reflexive, separable Banach space $(X,\|\cdot\|_X)$ let $X^*$ and $\|\cdot\|_{X^*}$ denote its dual space and the norm on the dual space respectively.
For $f\in X^*$ and $v\in X$ we use $\langle f,v \rangle$ to denote the duality pairing.
By $L^p(0,T;X)$ we denote the Lebesgue--Bochner space of 
equivalence classes of measurable functions 
with values in $X$ that satisfy
\begin{equation*}
\|x\|_{L^p(0,T;X)} := \left(\int_0^T\|x(t)\|_X^p dt)\right)^{1/p} < \infty.
\end{equation*}
By $L^p(\Omega; X)$ we denote the Lebesgue--Bochner space of random variables with values in $X$ and such that the norm
\begin{equation*}
\|x\|_{L^p(\Omega;X)} := \left(\E(\|x\|_X^p)\right)^{1/p}
\end{equation*}
is finite.
Finally by $\mathcal{L}^p(X)$ we denote the Lebesgue--Bochner space of 
$dt\times \P$-equivalence classes of 
$(\F_t)_{t\in[0,T]}$-adapted and $X$-valued stochastic process that satisfy
\begin{equation*}
\|x\|_{\mathcal{L}^p(X)} := \left( \E\int_0^T \|x(t)\|_X^p dt \right)^{1/p} < \infty.
\end{equation*}
We say that an operator $C:X\times \Omega \to X^*$ is measurable with respect to some $\mathcal{G} \subseteq \F$ if for any $v,w\in X$ the real-valued random variable $\langle C v,w\rangle$ is $\mathcal{G}$-measurable.

We assume that, with respect to $(\F_t)_{t\in [0,T]}$, 
$(W_t)_{t\in [0,T]}$ is a cylindrical $Q$-Wiener process with $Q=I$ on a separable Hilbert space $(U,(\cdot,\cdot)_U, |\cdot|_U)$.
We assume that there are $(V_1,\|\cdot\|_{V_1})$ and $(V_2,\|\cdot\|_{V_2})$, separable and reflexive Banach spaces that are densely and continuously embedded in $H$, where $(H,(\cdot,\cdot),|\cdot|)$ is a Hilbert space identified with its dual.
We thus have two Gelfand triples
\begin{equation*}
V_i \hookrightarrow H \hookrightarrow V_i^*\,, \,\,\, i \in \{1,2\}.
\end{equation*}
Let $A_i$ with $i\in\{1,2\}$ be operators defined on $V_i\times \Omega$
with values in $V_i^*$.
Let $B_i$ with $i\in\{1,2\}$ be operators defined on $V_i\times \Omega$
with values in $L_2(U,H)$,
where $L_2(U,H)$ is the space of Hilbert--Schmidt operators from $U$ to $H$.

Let $V:=V_1 \cap V_2$ and let the norm in $V$ be given by 
$\|\cdot\| := \|\cdot\|_{V_1} + \|\cdot\|_{V_2}$.
Assume that $V$ is separable and dense in both $V_1$ and $V_2$.
Using Gajewski, Gr\"oger and Zacharias~\cite[Kapitel I, Satz 5.13]{ggz}
one observes that the dual $V^*$ of $V$ can be identified with
\begin{equation*}
V_1^* + V_2^* := \{f = f_1 + f_2 : f_1 \in V_1^*,\, f_2\in V_2^*\}
\end{equation*}
and that for all $f \in V^*$
\begin{equation*}
\|f\|_{V*} = \inf \{\max(\|f_1\|_{V_1^*}, \|f_2\|_{V_2^*}): f = f_1 + f_2, f_1 \in V_1^*,\,\, f_2\in V_2^* \}.
\end{equation*}


We consider stochastic evolution equations of the form
\begin{equation}
\label{eq:see2}
du(t) = \big[A_1 u(t) + A_2 u(t)\big] dt + \big[B_1 u(t) + B_2 u(t)\big] dW(t)\, , \,\, t \in [0,T],
\end{equation}
where $u(0) = u_0$ with $u_0$ a given $H$-valued and $\F_0$-measurable random variable.
Let $A:=A_1 + A_2$ and $B:=B_1 + B_2$.
The operator $A$ is defined on $V\times\Omega$ with values in $V^*$
and the operator $B$ is defined on $V\times\Omega$ with values in $L_2(U,H)$.
Then we can write~\eqref{eq:see2} as~\eqref{eq:see}.



We impose the following assumptions on the operators.

\begin{assumption}
\label{ass:operators}
Let $A_i : V_i \times \Omega \to V_i^*$ 
be $\F_0$-measurable operators for $i\in \{1,2\}$.
Let $B_i: V_i \times \Omega \to L_2(U,H)$ be such that 
for any $v\in V_i$, $u\in U$ and $h\in H$ the real-valued 
random variable $((B_i v)u,h)$ is $\F_0$-measurable for $i\in \{1,2\}$.
Moreover assume that the following conditions hold. 

{\em Monotonicity:}
\begin{equation*}
2\langle A v - A w, v-w\rangle + \|B v-B w\|_{L_2(U,H)}^2 \leq K|v-w|^2\,\, \textrm{ for all }\,\, v,w \in V.
\end{equation*}

{\em Coercivity:} there is $\mu > 0$ such that
\begin{equation*}
2\langle A_1 v, v \rangle + \|B_1  v\|_{L_2(U,H)}^2 \leq - \mu \|v\|_{V_1}^2 + K(1+|v|^2)\,\, \textrm{ for all }\,\, v\in V_1.
\end{equation*}
and
\begin{equation*}
2\langle A_2 v, v \rangle + \|B_2  v\|_{L_2(U,H)}^2 \leq K(1+|v|^2)\,\, \textrm{ for all }\,\, v\in V_2.
\end{equation*}

{\em Growth:}
\begin{equation*}
\|A_1 v\|_{V_1^*}^2 \leq K(1+\|v\|_{V_1}^2)\,\, \textrm{ for all }\,\, v\in V_1
\end{equation*}
and
\begin{equation*}
\|A_2 v\|_{V_2^*}^{p^*} \leq K(1+\|v\|_{V_2}^p)\,\, \textrm{ for all }\,\, v\in V_2.
\end{equation*}
as well as
\begin{equation*}
\|Bv\|_{L_2(U,H)}^2 \leq K(1+|v|^2) \,\, \textrm{ for all }\,\, v\in H.
\end{equation*}
%

{\em Hemicontinuity:} for any $v,w$ and $z$ in $V$
\begin{equation*}
\langle A(v+\epsilon w),z\rangle \to \langle Av,z \rangle \,\,\textrm{ as }\,\, \epsilon \to 0.
\end{equation*}
\end{assumption}

We now define what is meant by solution of~\eqref{eq:see}.
\begin{definition}[Solution]
\label{def:soln}
Let $u_0$ be an $\F_0$-measurable $H$-valued random variable.
We say that a continuous, $H$-valued and $(\F_t)_{t\in [0,T]}$-adapted process $u$ is a solution to~\eqref{eq:see}
if $u$ is $dt \times \P$ almost everywhere $V$-valued,
if $u\in \mathcal{L}^2(V_1)\cap \mathcal{L}^p(V_2)$
and if for every $t\in [0,T]$ and every $v\in V$, almost surely,
\begin{equation*}
( u(t),v ) = ( u_0, v ) + \int_0^t \langle A u(s), v \rangle ds + \int_0^t (v,Bu(s) dW(s)).
\end{equation*}
\end{definition}
To the best knowledge of the authors, existence and uniqueness has not been proved for this class of stochastic evolution equations.
Pardoux~\cite{pardoux:thesis} considers the situation where the stochastic evolution equation is governed by a sum of monotone, coercive and hemicontinuous operators satisfying certain growth condition.
However the operator $A_2$ in our case only satisfies a type of ``degenerate'' coercivity condition.
Hence the existence theorem from Pardoux~\cite{pardoux:thesis} does not apply.
We prove that a solution to~\eqref{eq:see} must be unique 
in Theorem~\ref{thm:uniqueness} and 
we prove existence of the solution 
in Theorem~\ref{thm:main}.

\begin{theorem}[Uniqueness]
\label{thm:uniqueness}
The solution of~\eqref{eq:see}, specified by Definition~\ref{def:soln}, is unique, provided that the Growth and Monotonicity conditions in Assumption~\ref{ass:operators} are satisfied.
\end{theorem}

We prove Theorem~\ref{thm:uniqueness} in Section~\ref{sec:convergence}. 
Let us now describe the discretization scheme for the stochastic evolution
equation~\eqref{eq:see2}.
For the space discretization let $(V_m)_{m \in \N}$ be a Galerkin scheme for $V$.
To be precise we assume that $V_m\subseteq V$
are finite dimensional spaces with the dimension of $V_m$ equal to $m$. 
We further assume that $V_m \subseteq V_{m+1}$ for all $m\in \N$ and that
\begin{equation*}
\lim_{m\to \infty} \inf\{\|v-\varphi\|\, :\, \varphi \in V_m\}=0\,\,\,\, \forall v\in V.
\end{equation*}
(this is referred to as the limited completeness of the Galerkin scheme).
We will need the following projection operators.

\begin{assumption}
\label{ass:projection}
For any $m\in \N$ let $\Pi_m:V^* \to V_m$ satisfy the following:
\begin{enumerate}
\item For any $v \in V_m$, $\Pi_m v = v$.
\item If $f\in V^*$ and $v\in V$ then $\langle f, \Pi_m v \rangle = \langle v, \Pi_m f \rangle$.
\item If $g,h \in H$ then $(\Pi_m g,h) = (\Pi_m h, g)$
and $|\Pi_m h| \leq |h|$.
\item There is a constant, depending on $m$ and denoted by $\mathfrak{c}(m)$, such that
\begin{equation*}
|\Pi_m f|^2 \leq \mathfrak{c}(m)\|f\|_{V^*}^2\,\, \textrm{ for all }\,\, f\in V^*.
\end{equation*}
\end{enumerate}
\end{assumption}
In applications this assumption is easily satisfied.
In particular if $\{\varphi_j \in V\,\, : \,\, j = 1,2,\ldots \}$
is an orthonormal basis in $H$ then taking
$V_m := \textrm{span}\{\varphi_1,\ldots,\varphi_m\}$
is a Galerkin scheme for $V$.
Taking $\Pi_m f := \sum_{j=1}^m \langle f, \varphi_j \rangle \varphi_j$
satisfies the first three conditions in Assumption~\ref{ass:projection}.
Moreover, the following holds 
\begin{equation*}
|\Pi_m f|^2 = \bigg| \sum_{j=1}^m \langle f, \varphi_j\rangle \varphi_j \bigg|^2 = \sum_{j=1}^m \langle f, \varphi_j\rangle^2 \leq \|f\|_{V^*}^2 \sum_{j=1}^m \|\varphi_j\|_V^2 = \mathfrak{c}(m) \|f\|_{V^*}^2,
\end{equation*}
where $\mathfrak{c}(m) := \sum_{j=1}^m \|\varphi_j\|_V^2$.
Thus the fourth condition in Assumption~\ref{ass:projection} is also satisfied.
Let $\{\chi_i\}_{i\in \N}$ be an orthonormal basis of $U$.
Fix $k \in \N$ and define
\begin{equation*}
W_k(t) := \sum_{j=1}^{k} (W(t),\chi_j)_U \chi_j.
\end{equation*}
For the time discretization take $n\in \N$, let $\tau_n := T/n$ and define the grid points on an equidistant grid as $t^n_i := \tau_n i$, $i=0,1,\ldots,n$.
Further consider some sequence $((n_\ell, m_\ell, k_\ell))_{\ell \in \N}$ such that $n_\ell, m_\ell$ and $k_\ell$ all go to infinity as $\ell \to \infty$.

Let $c$ denote a generic positive constant that may depend on $T$, 
on the constants arising in the continuous embeddings $V_i \hookrightarrow H \hookrightarrow V_i^*$, $i=1,2$ 
and on the constants arising in Assumptions~\ref{ass:operators} 
and~\ref{ass:interpolation} but that is always independent of the discretization parameters $m$, $k$ and $n$.
Define $\kappa_{n_\ell}(t) = t^{n_\ell}_i$ if $t\in [t^{n_\ell}_i, t^{n_\ell}_{i+1})$ for $i=0,\ldots,n_\ell-1$ and $\kappa_{n_\ell}(T) = T$.
Fix some $\ell \in \N$ (and hence $m_\ell$, $n_\ell$ and $k_\ell$).
Let $u_\ell(0)$ be a $V_{m_\ell}$ valued $\F_0$-measurable approximation of $u_0$.
For example we can take $u_\ell(0) := \Pi_{m_\ell} u_0$ but other approximations are possible.
For $t>0$ we define a process $u_\ell$ by
\begin{equation}
\label{eq:disc1}
\begin{split}
u_\ell(t) = u_\ell(0) & + \int_0^t \Pi_{m_\ell} \left[A_1 u_\ell(\kappa_{n_\ell}(s))
+ A_{2,\ell} u_\ell(\kappa_{n_\ell}(s)) \right] ds\\
& + \int_0^t \Pi_{m_\ell} B u_\ell(\kappa_{n_\ell}(s))dW_{k_\ell}(s),
\end{split}
\end{equation}
where we use the ``tamed'' operator $A_{2,\ell}$ defined by
\begin{equation}
\label{eq:def:taming}
A_{2,\ell} v := \frac{1}{1+{n_\ell}^{-1/2}|\Pi_{m_\ell} A_2 v|}A_2 v
\end{equation}
for any $v \in V_2$.
We will use the following notation:
$\bar{u}_\ell(t) := u_\ell(\kappa_{n_\ell}(t))$
and $a_\ell(v) := \Pi_{m_\ell}[A_1 v + A_{2,\ell} v]$.
Then~\eqref{eq:disc1} is equivalent to
\begin{equation}
\label{eq:disc2}
\begin{split}
u_\ell(t) = u_\ell(0) & + \int_0^t a_\ell(\bar{u}_\ell(s))
ds + \int_0^t \Pi_{m_\ell} B \bar{u}_\ell(s) dW_{k_\ell}(s).
\end{split}
\end{equation}
In particular at the time-grid points we have, for $i=0,1,\ldots,n_\ell-1$,
\begin{equation*}
\begin{split}
u_\ell(t^{n_\ell}_{i+1}) = u_\ell(t^{n_\ell}_i) & + a_\ell(u_\ell(t^{n_\ell}_i))\tau_{n_\ell} + \Pi_{m_\ell} B u_\ell(t^{n_\ell}_i) \Delta W_{k_\ell}(t_{i+1}),
\end{split}
\end{equation*}
where $\Delta W_{k_\ell}(t^{n_\ell}_{i+1}) := W_{k_\ell}(t^{n_\ell}_{i+1})-W_{k_\ell}(t^{n_\ell}_i)$.
This in turn is equivalent to
\begin{equation*}
\begin{split}
\frac{u_\ell(t^{n_\ell}_{i+1}) - u_\ell(t^{n_\ell}_i)}{\tau_{n_\ell}} = & a_\ell(u_\ell(t^{n_\ell}_i)) + \Pi_{m_\ell} Bu_\ell(t^{n_\ell}_i) \frac{\Delta W_{k_\ell}(t^{n_\ell}_{i+1})}{\tau_{n_\ell}}.
\end{split}
\end{equation*}

We list below the properies which are satisfied by the tamed operator $A_{2,\ell}$. 
These are consequences of its structure and the assumed properties of $A_2$.
For brevity let, for any $v\in V_2$,
\begin{equation}
\label{eq:def_of_T_ell}
T_\ell(v) := \frac{1}{1+n_\ell^{-1/2}|\Pi_{m_\ell} A_2 v|}.
\end{equation}
Then for any $v\in V_2$,
\begin{equation}
\label{eq:taming_less_n}
|\Pi_{m_\ell} A_{2,\ell} v| = T_\ell(v)|\Pi_{m_\ell} A_2 v| \leq n_\ell^{1/2}
\end{equation}
and also, using the Growth assumption on $A_2$,
\begin{equation}
\label{eq:taming_less_original}
\|A_{2,\ell} v\|_{V_2^*}^{p^*} = T_\ell(v)^{p^*}\|A_2 v\|_{V_2^*}^{p^*} \leq K(1+\|v\|_{V_2}^p).
\end{equation}
Furthermore, using the Coercivity assumption on $A_2$, we note that for all $v \in V_{m_\ell}$ we have
\begin{equation}
\label{eq:taming_retains_weak_coerc}
2\langle A_{2,\ell} v, v \rangle = 2T_\ell(v)\langle A_2 v,v \rangle \leq K(1+|v|^2).
\end{equation}
Thus the weaker coercivity assumption that has been made about $A_2$ is retained.
Consider, for a moment, that $A_2$ satisfies the ``usual'' coercivity condition
\begin{equation*}
2\langle A_2 v, v \rangle + \|B_2  v\|_{L_2(U,H)}^2 \leq - \mu \|v\|_{V_2}^p + K(1+|v|^2)\,\, \textrm{ for all }\,\, v\in V_2.
\end{equation*}
We see that in this case the best coercivity we can get from this for $A_{2,\ell}$ is again only~\eqref{eq:taming_retains_weak_coerc}.
Hence to obtain the necessary a priori estimates
we will need an interpolation inequality between $V_2$ and $V_1$ with $H$.

\begin{assumption}
\label{ass:interpolation}
There are constants $\lambda \in [0,2/p)$  and
$\Lambda>0$ such that for any $v\in V$
\begin{equation*}
\|v\|_{V_2} \leq \Lambda  \|v\|_{V_1}^\lambda |v|^{1-\lambda}. 
\end{equation*}
\end{assumption}
Note that in order to overcome the difficulty with coercivity it would suffice to have Assumption~\ref{ass:interpolation} satisfied with $\lambda \in [0,2/p]$.
However monotonicity of $A_2$ is not preserved by taming.
To overcome this we will need to show that $A_{2,\ell} \bar{u}_\ell - A_2 \bar{u}_\ell \to 0$ in $\mathcal{L}^{p*}(V_2^*)$.
To achieve this we use the fact that $\lambda \in [0,2/p)$ in Lemma~\ref{lemma:strong_conv_of_tam_correction} and the following observation:
Assumption~\ref{ass:interpolation} implies that there is $\eta > 0$ such that
\begin{equation*}
\|v\|_{V_2}^{p(1+\eta)} \leq c \|v\|_{V_1}^2 |v|^\rho,
\end{equation*}
where $\rho := (1-\lambda)p(1+\eta)$.
From this it follows that if $v\in L^2(\Omega; L^2(0,T;V_1))$ and $v\in L^{2\rho}(\Omega; L^\infty(0,T;H))$ then
\begin{equation*}
\E\int_0^T \|v(s)\|_{V_2}^{p(1+\eta)} ds \leq c\Bigg[ \E \sup_{s\leq t}|v(s)|^{2\rho} + \E\bigg(\int_0^T \|v(s)\|_{V_1}^2 ds\bigg)^2 \Bigg].
\end{equation*}
Thus we see that Assumption~\ref{ass:interpolation} allows us to control the approximate solution in the $L^{p(1+\eta)}((0,T)\times \Omega; V_2)$ norm,
provided that we can control the approximate solution in the norms of $L^2(\Omega; L^2(0,T;V_1))$
and $L^{2\rho}(\Omega; L^\infty(0,T;H))$.

Let us take $q_0 := \max(4,2\rho)$.
Now we can state the main result of this paper.
\begin{theorem}
\label{thm:main}
Let Assumptions~\ref{ass:operators},~\ref{ass:projection} and~\ref{ass:interpolation} be satisfied.
Let $u_0 \in L^{q_0}(\Omega; H)$ and let $u_\ell(0) \to u_0$ in $L^{q_0}(\Omega; H)$.
Assume that $\tfrac{\mathfrak{c}(m_\ell)}{n_{\ell}} \to 0$ as $\ell \to \infty$.
Then there exists a unique solution $u$ to~\eqref{eq:see} and
$\bar{u}_\ell \weaklyto u$ in $\mathcal{L}^2(V_1)$ and in $\mathcal{L}^p(V_2)$ and $u_\ell(T) \to u(T)$ in $L^2(\Omega;H)$  as $\ell \to \infty$.
\end{theorem}

In Section~\ref{sec:examples} we provide examples of stochastic partial
differential equations where Theorem~\ref{thm:main} can be applied.
We also compute $\mathfrak{c}(m)$ in case of the spectral Galerkin method to 
make the implications of the space-time coupling constraint more explicit. 
The crucial point is that the requirement is no more onerous than in the case
of equations with operators growing at most linearly.

\section{A priori estimates}
\label{sec:apriori_est}

We start with an important observation that allows us to use standard results on bounds of stochastic integrals driven by finite dimensional Wiener processes.

\begin{remark}
\label{remark:n_bound}
Recall that $(\chi_j)_{j\in \N}$ is an orthonormal basis in $U$.
Moreover recall that $\bar{u}_\ell(t) := u_\ell(\kappa_{n_\ell}(t))$
and that $a_\ell(v) := \Pi_{m_\ell}[A_1 v + A_{2,\ell} v]$.
For each $j\in \N$ a Wiener processes $\mathcal{W}_j$ 
is obtained by taking $\mathcal{W}_j(t) := (W(t),\chi_j)_U$.
If $i\neq j$ then $\mathcal{W}_i$ and $\mathcal{W}_j$ are independent.
Furthermore~\eqref{eq:disc2} is equivalent to
\begin{equation*}
u_\ell(t) = u_\ell(0) + \int_0^t a_\ell(\bar{u}_\ell(s))ds
+ \sum_{j=1}^{k_\ell} \int_0^t \Pi_{m_\ell} B \bar{u}_\ell(s) \chi_j d\mathcal{W}_j(s).
\end{equation*}
Fix $\ell\in \N$ (and thus $k_\ell, m_\ell$ and $n_\ell$ are alse fixed).
Then using the Growth assumptions on $A_1$ and $B$, Assumption~\ref{ass:projection}
as well as~\eqref{eq:taming_less_n} one observes that
$|a_\ell(v)|^2 \leq 2\mathfrak{c}(m)K(1+|v|^2) + 2n_\ell$ and
$|\Pi_{m_\ell}B v\chi_j|^2 \leq K(1+|v|^2)$.
Hence one knows that, for $q\geq 1$,
\begin{equation*}
\E\sup_{0\leq t \leq T}  |u_\ell(t)|^q < \infty,
\end{equation*}
provided that $\E|u_\ell(0)|^q < \infty$.
Clearly, at this point, one cannot claim that this bound is independent of $\ell$.

\end{remark}

One applies It\^o's formula to~\eqref{eq:disc2} to obtain
\begin{equation*}
\begin{split}
|u_\ell(t)|^2 = |u_\ell(0)|^2 & + \int_0^t \!\! \bigg[2 \langle a_\ell(\bar{u}_\ell(s)),u_\ell(s)\rangle + \sum_{j=1}^{k_\ell}|\Pi_{m_\ell} B \bar{u}_\ell(s)\chi_j|^2\bigg] ds\\
& + \int_0^t 2 (B \bar{u}_\ell(s), u_\ell(s) dW_{k_\ell}(s))\,,
\end{split}
\end{equation*}
which can be rewritten as 
\begin{equation}
\label{eq:apriori_start}
\begin{split}
|u_\ell(t)|^2 = |u_\ell(0)|^2 & + \int_0^t \!\! \bigg[2 \langle a_\ell(\bar{u}_\ell(s)),\bar{u}_\ell(s)\rangle + \sum_{j=1}^{k_\ell}|\Pi_{m_\ell} B \bar{u}_\ell(s)\chi_j|^2 \bigg] ds\\
& + 2\int_0^t \!\!  \langle a_\ell(\bar{u}_\ell(s)),u_\ell(s)-\bar{u}_\ell(s)\rangle ds\\
& + \int_0^t 2 (B \bar{u}_\ell(s), u_\ell(s) dW_{k_\ell}(s))\, ,
\end{split}
\end{equation}
in order to apply the coercivity assumption so as to obtain the a priori estimates for the discretized equation.

First we concentrate on the term that arises from the ``correction'' that one has to make to use the coercivity assumption due to the use of an explicit scheme.

\begin{lemma}
\label{lemma:corr_term_est}
Let the Growth condition in Assumption~\ref{ass:operators} be satisfied.
Let Assumption~\ref{ass:projection} hold.
Let $q\geq 1$ be given.
Then
\begin{equation}
\label{eq:lemma_corr_term_est1}
\begin{split}
\E  \bigg( & \frac{1}{\tau_{n_\ell}} \int_0^t  |u_\ell(s)-\bar{u}_\ell(s)|^2 ds \bigg)^q\\
& \leq c_{T,q} \Bigg(1+ (\mathfrak{c}(m)\tau_{n_\ell})^q \E\bigg(\int_0^t   \|\bar{u}_\ell(s)\|_{V_1}^2 ds \bigg)^q + \E\int_0^t |\bar{u}_\ell(s)|^{2q}ds \Bigg)
\end{split}
\end{equation}
and
\begin{equation}
\label{eq:lemma_corr_term_est2}
\begin{split}
\E & \bigg( \int_0^t  |a_\ell(\bar{u}_\ell(s))||u_\ell(s)-\bar{u}_\ell(s)|ds\bigg)^q \\
& \leq c_{T,q}\Bigg(1+ (\mathfrak{c}(m)\tau_{n_\ell})^q \E\bigg(\int_0^t   \|\bar{u}_\ell(s)\|_{V_1}^2 ds \bigg)^q + \E\int_0^t |\bar{u}_\ell(s)|^{2q}ds \Bigg).
\end{split}
\end{equation}
\end{lemma}

\begin{proof}
From~\eqref{eq:disc1} it is clear that
\begin{equation*}
\begin{split}
& I_{1,\ell}(t)  := \E  \bigg( \frac{1}{\tau_{n_\ell}} \int_0^t |u_\ell(s)-\bar{u}_\ell(s)|^2 ds \bigg)^q \\
& = \E\bigg( \int_0^t \frac{1}{\tau_{n_\ell}}\bigg| \int_{\kappa_{n_\ell}(s)}^s a_\ell(\bar{u}_\ell(r)) dr + \int_{\kappa_{n_\ell}(s)}^s \Pi_{m_\ell} B \bar{u}_\ell(r) dW_{k_\ell}(r)  \bigg|^2 ds \bigg)^q\\
& \leq 2^q \E\bigg( \frac{1}{\tau_{n_\ell}}\int_0^t \bigg| \int_{\kappa_{n_\ell}(s)}^s \!\!\! a_\ell(\bar{u}_\ell(r)) dr \bigg|^2
+ \bigg|\int_{\kappa_{n_\ell}(s)}^s \Pi_{m_\ell} B \bar{u}_\ell(r) dW_{k_\ell}(r)  \bigg|^2 \!\!ds \bigg)^q\!\!.
\end{split}
\end{equation*}
Applying H\"older's inequality yields
\begin{equation*}
\begin{split}
 I_{1,\ell}(t) & \leq  2^q \E\Bigg( \frac{1}{\tau_{n_\ell}}\int_0^t\Bigg[ (s-\kappa_{n_\ell}(s)) \int_{\kappa_{n_\ell}(s)}^s |a_\ell(\bar{u}_\ell(r))|^2 dr \\
& \quad\quad\quad \quad\quad\quad + \bigg|\int_{\kappa_{n_\ell}(s)}^s \Pi_{m_\ell} B \bar{u}_\ell(r) dW_{k_\ell}(r)  \bigg|^2 \Bigg]ds \Bigg)^q\\
& \leq c_q \E\bigg( \frac{1}{\tau_{n_\ell}}\int_0^t \tau_{n_\ell}^2 |a_\ell(\bar{u}_\ell(s))|^2 ds \bigg)^q \\
& \quad\quad\quad \quad\quad\quad + c_q\E\bigg( \frac{1}{\tau_{n_\ell}}\int_0^t \bigg|\int_{\kappa_{n_\ell}(s)}^s \Pi_{m_\ell} B \bar{u}_\ell(r) dW_{k_\ell}(r)  \bigg|^2 ds \bigg)^q.
\end{split}
\end{equation*}
Using Assumption~\ref{ass:projection} and~\eqref{eq:taming_less_n}  one obtains
\begin{equation}
\label{eq:corr_term_est_2}
\begin{split}
& I_{1,\ell}(t)  \leq c_q \E\bigg(\int_0^t \tau_{n_\ell} [ 2\mathfrak{c}(m)\|A_1 \bar{u}_\ell(s)\|_{V_1^*}^2 + 2n_\ell] ds \bigg)^q \\
&  + c_q\E\bigg( \int_0^t \frac{1}{\tau_{n_\ell}} \bigg|\int_{\kappa_{n_\ell}(s)}^s \Pi_{m_\ell} B \bar{u}_\ell(r) dW_{k_\ell}(r)  \bigg|^2 ds \bigg)^q
:= I_{2,\ell}(t) + I_{3,\ell}(t).
\end{split}
\end{equation}
The Growth assumption on $A_1$ implies that
\begin{equation*}
I_{2,\ell}(t) \leq c_{T,q}\bigg(1 + (\mathfrak{c}(m)\tau_{n_\ell})^q\E\bigg(\int_0^t \|\bar{u}_\ell(s)\|_{V_1}^2 ds \bigg)^q \bigg).
\end{equation*}
Using H\"older's inequality leads to
\begin{equation*}
I_{3,\ell}(t) \leq c_{T,q} \E\int_0^t \frac{1}{\tau_{n_\ell}^q} \bigg|\int_{\kappa_{n_\ell}(s)}^s \Pi_{m_\ell} B \bar{u}_\ell(r) dW_{k_\ell}(r)  \bigg|^{2q} ds.
\end{equation*}
Due to Remark~\ref{remark:n_bound} and the Growth assumption on $B$ one observes that
\begin{equation*}
\begin{split}
 I_{3,\ell}(t) & \leq c_{T,q} \int_0^t \frac{1}{\tau_{n_\ell}^q} \E \bigg|\int_{\kappa_{n_\ell}(s)}^s \|B \bar{u}_\ell(r)\|_{L_2(U,H)}^2 dr  \bigg|^q ds \\
& \leq c_{T,q}\E  \int_0^t \frac{1}{\tau_{n_\ell}^q} (s-\kappa_{n_\ell}(s))^q \|B \bar{u}_\ell(s)\|_{L_2(U,H)}^{2q}  ds\\
& \leq c_{T,q}\bigg(1+\E\int_0^t |\bar{u}_\ell(s)|^{2q} ds \bigg).
\end{split}
\end{equation*}
This implies~\eqref{eq:lemma_corr_term_est1}.
Moreover Assumption~\ref{ass:projection} and~\eqref{eq:taming_less_n} imply that
\begin{equation*}
\begin{split}
I_\ell(t) & :=  \E\bigg(\int_0^t |a_\ell(\bar{u}_\ell(s))||u_\ell(s)-\bar{u}_\ell(s)|ds\bigg)^q \\
& \leq 2^q \E\bigg(\int_0^t \tau_{n_\ell}|a_\ell(\bar{u}_\ell(s))|^2 + \frac{1}{\tau_{n_\ell}}|u_\ell(s)-\bar{u}_\ell(s)|^2 ds\bigg)^q\\
& \leq c_q \E\bigg(\int_0^t [\mathfrak{c}(m)\tau_{n_\ell}\|A_1 \bar{u}_\ell(s)\|_{V_1^*}^2 + \tau_{n_\ell} n_\ell]ds \bigg)^q \\
& \quad\quad\quad \quad\quad\quad + c_q \E\bigg(\frac{1}{\tau_{n_\ell}} \int_0^t |u_\ell(s)-\bar{u}_\ell(s)|^2 ds\bigg)^q.
\end{split}
\end{equation*}
Applying the Growth assumption on $A_1$ yields
\begin{equation}
\label{eq:corr_term_est1}
\begin{split}
I_\ell(t) & \leq c_{T,q}\bigg(1 + (\mathfrak{c}(m)\tau_{n_\ell})^q \E\bigg(\int_0^t \|\bar{u}_\ell(s)\|_{V_1}^2 ds \bigg)^q \bigg) \\
& \quad\quad\quad \quad\quad\quad + c_q \E\bigg(\frac{1}{\tau_{n_\ell}} \int_0^t |u_\ell(s)-\bar{u}_\ell(s)|^2 ds\bigg)^q.
\end{split}
\end{equation}
Using~\eqref{eq:lemma_corr_term_est1} in~\eqref{eq:corr_term_est1} concludes the proof.
\end{proof}

\begin{theorem}[A priori estimate]
\label{thm:apriori2}
Let the Coercivity and Growth conditions in Assumption~\ref{ass:operators} hold.
Let Assumption~\ref{ass:projection} be satisfied.
Let $q \geq 1$ be given
and assume that $\E|u_\ell(0)|^{2q} < c$ and that $u_\ell(0)$ is $\F_0$-measurable.
There is $\epsilon \in (0,\infty)$ such that for all 
$\ell \in \{\ell' \in \N: \mathfrak{c}(m_{\ell'})\tau_{n_{\ell'}} < \epsilon\}$ we have,
for any $t\in [0,T]$,
\begin{equation*}
\E \sup_{s\in [0,t]}|u_\ell(s)|^{2q} + \mu^{q}  \E \bigg(\int_0^t \|\bar{u}_\ell(s)\|_{V_1}^2 ds \bigg)^{q}
\leq c\left(1+ \E|u_\ell(0)|^{2q}\right).
\end{equation*}
\end{theorem}

\begin{proof}
Applying the Coercivity assumption in~\eqref{eq:apriori_start},
raising to power $q \geq 1$,
taking the supremum over $s\leq t$ and taking the expectation yields
\begin{equation}
\label{eq:higher_moments_proof1}
\begin{split}
& \E  \sup_{s\leq t}  |u_\ell(s)|^{2q}  + \mu^{q} \E \bigg( \int_0^t \|\bar{u}_\ell(s)\|_{V_1}^2 ds \bigg)^{q} \leq c_{T,q} \bigg[1 + \E|u_\ell(0)|^{2q} \\
& + \E \bigg( \int_0^t |\bar{u}_\ell(s)|^2 ds \bigg)^{q}
  + \E\bigg(\int_0^t |a_\ell(\bar{u}_\ell(s))||u_\ell(s)-\bar{u}_\ell(s)|ds\bigg)^{q}\\
& + \E \sup_{s \leq t}\bigg|\int_0^s (u_\ell(s), B\bar{u}_\ell(s) dW_{k_\ell}(s))\bigg|^{q} \bigg].
\end{split}
\end{equation}
Using Lemma~\ref{lemma:corr_term_est} in~\eqref{eq:higher_moments_proof1} results in
\begin{equation}
\label{eq:higher_moments_proof2}
\begin{split}
& \E  \sup_{s\leq t}  |u_\ell(s)|^{2q}  + \frac{\mu^{q}}{2} \E \bigg( \int_0^t \|\bar{u}_\ell(s)\|_{V_1}^2 ds \bigg)^{q}\\
& \leq c_{T,q} \bigg[1 + \E|u_\ell(0)|^{2q} + \E \bigg( \int_0^t |\bar{u}_\ell(s)|^2 ds \bigg)^{q} \\
& \quad + \E \int_0^t |\bar{u}_\ell(s)|^{2q} ds
+ \E \sup_{s \leq t}\bigg(\int_0^s (u_\ell(r), B\bar{u}_\ell(r) dW_{k_\ell}(r))\bigg)^{q}\bigg].
\end{split}
\end{equation}
Using Burkholder--Davis--Gundy inequality one obtains
\begin{equation*}
\begin{split}
I_\ell & := c_{T,q}\E \sup_{s \leq t}\bigg|\int_0^s (u_\ell(r), B\bar{u}_\ell(r) dW_{k_\ell}(r))\bigg|^{q} \\
& \leq c_{T,q} \E \bigg|\int_0^t |u_\ell(s)|^2 \|B\bar{u}_\ell(s)\|_{L_2(U,H)}^2 ds\bigg|^{q/2}\\
& \leq c_{T,q} \E \Bigg[\sup_{s\leq t}|u_\ell(s)|^{q}\bigg(\int_0^t \|B\bar{u}_\ell(s)\|_{L_2(U,H)}^2ds\bigg)^{q/2} \Bigg].
\end{split}
\end{equation*}
Young's inequality and the Growth assumption on $B$ imply that
\begin{equation*}
\begin{split}
I_\ell & \leq \frac{1}{2}\E \sup_{s\leq t}|u_\ell(s)|^{2q} + c \E \bigg(\int_0^t \|B\bar{u}_\ell(s)\|_{L_2(U,H)}^2ds\bigg)^{q}\\
& \leq \frac{1}{2}\E \sup_{s\leq t}|u_\ell(s)|^{2q} + c\bigg(1+\int_0^t \E \sup_{r\leq s}|u_\ell(r)|^{2q} ds \bigg).
\end{split}
\end{equation*}
Applying this in~\eqref{eq:higher_moments_proof2} leads to
\begin{equation*}
\begin{split}
\frac{1}{2}\E  \sup_{s\leq t}  |u_\ell(s)|^{2q}  +  \frac{\mu^{q}}{2} \E \bigg( \int_0^t \|\bar{u}_\ell(s)\|_{V_1}^2 ds \bigg)^{q}
& \leq  c \bigg[1 + \E|u_\ell(0)|^{2q} \\
& + \int_0^t \E \sup_{r\leq s}|u_\ell(r)|^{2q} ds\bigg].
\end{split}
\end{equation*}
Application of Gronwall's lemma yields
\begin{equation}
\label{eq:higher_moments_proof3}
\E \sup_{s\in [0,t]}|u_\ell(s)|^{2q} + \mu^{q}  \E \bigg(\int_0^t \|\bar{u}_\ell(s)\|_{V_1}^2 ds \bigg)^{q}
\leq c \left(1+\E|u_\ell(0)|^{2q}\right).
\end{equation}
\end{proof}

Now we use Theorem~\ref{thm:apriori2} and Assumption~\ref{ass:interpolation} to obtain the remaining required estimates.

\begin{corollary}[Remaining a priori estimates]
\label{corollary:remaining_apriori}
Let the Growth and Coercivity conditions in Assumption~\ref{ass:operators} be satisfied.
Let Assumptions~\ref{ass:projection} and~\ref{ass:interpolation} hold.
Let $u_\ell(0)$ be bounded in $L^{q_0}(\Omega; H)$, uniformly with respect to $\ell$.
There is $\epsilon \in (0,\infty)$ such that for all 
$\ell \in \{\ell' \in \N: \mathfrak{c}({m_\ell'}\tau_{n_{\ell'}} < \epsilon\}$ we have
\begin{equation}
\label{eq:remaining_apriori1}
\E\int_0^T \|A_1 \bar{u}_\ell\|_{V_1^*}^2 ds \leq c,\quad \E\int_0^T \|B \bar{u}_\ell\|_{L_2(U,H)}^2 ds \leq c.
\end{equation}
Furthermore
\begin{equation}
\label{eq:remaining_apriori4}
\E\int_0^T |u_\ell(s) - \bar{u}_\ell(s)|^2 ds \leq c \tau_{n_\ell}.
\end{equation}
Finally, for some $\eta > 0$,
\begin{equation}
\label{eq:remaining_apriori2}
\E\int_0^T \|\bar{u}_\ell\|_{V_2}^{p(1+\eta)} ds \leq c
\end{equation}
and
\begin{equation}
\label{eq:remaining_apriori3}
\E\int_0^T \|A_2 \bar{u}_\ell\|_{V_2^*}^{p^*(1+\eta)} ds \leq c,\quad \E\int_0^T \|A_{2,\ell} \bar{u}_\ell\|_{V_2^*}^{p^*(1+\eta)} ds \leq c.
\end{equation}
\end{corollary}

\begin{proof}
Inequality~\eqref{eq:remaining_apriori1} follows directly from the Growth assumptions on $A_1$ and $B$ and from Theorem~\ref{thm:apriori2} with $q=1$.
Using~\eqref{eq:lemma_corr_term_est1}, 
together with Theorem~\ref{thm:apriori2} with $q=1$,
yields~\eqref{eq:remaining_apriori4}.
Since $u_\ell(0)$ is assumed to be bounded in $L^{q_0}(\Omega; H)$,
uniformly in $\ell$,
one can conclude, using Theorem~\ref{thm:apriori2}, that
\begin{equation*}
\E \sup_{s\in [0,t]}|u_\ell(s)|^{2\rho} \leq c \,\,\textrm{ and }\,\, \E \bigg(\int_0^t \|\bar{u}_\ell(s)\|_{V_1}^2 ds \bigg)^2 \leq c.
\end{equation*}
This, together with Assumption~\ref{ass:interpolation}, yields~\eqref{eq:remaining_apriori2}.
Finally,~\eqref{eq:remaining_apriori2}, the assumption on the growth of $A_2$ and~\eqref{eq:taming_less_original} lead to~\eqref{eq:remaining_apriori3}.
\end{proof}

\section{Convergence}
\label{sec:convergence}
Having obtained the required a priori estimates we can use compactness arguments to extract weakly convergent subsequences of the approximation.

\begin{lemma}
\label{lemma:weak_limits_from_compactness}
Let the Growth and Coercivity conditions in Assumption~\ref{ass:operators} hold.
Let Assumptions~\ref{ass:projection} and~\ref{ass:interpolation} be satisfied.
Let $u_\ell(0) \to u_0$ in $L^{q_0}(\Omega; H)$.
Let $\tfrac{\mathfrak{c}(m_\ell)}{n_\ell} \to 0$ as $\ell \to \infty$.
Then there is a subsequence of the sequence $\ell$, which we denote $\ell'$, and $u \in \mathcal{L}^2(V_1) \cap \mathcal{L}^p(V_2)$ such that, as $\ell' \to \infty$,
\begin{equation*}
\bar{u}_{\ell'} \weaklyto u\,\, \textrm{ in }\,\, \mathcal{L}^2(V_1)\,\, \textrm{ and in }\,\, \mathcal{L}^p(V_2).
\end{equation*}
Furthermore there are $a_1^\infty \in \mathcal{L}^2(V_1^*)$, $a_2^\infty \in \mathcal{L}^{p^*}(V_2^*)$ and $b^\infty \in \mathcal{L}^2(L_2(U,H))$ such that,  as $\ell' \to \infty$,
\begin{equation*}
A_1\bar{u}_{\ell'} \weaklyto a_1^\infty \,\, \textrm{ in }\,\, \mathcal{L}^2(V_1^*),\,\, A_{2,\ell'}\bar{u}_{\ell'} \weaklyto a_2^\infty \,\, \textrm{ in }\,\, \mathcal{L}^{p^*}(V_2^*)
\end{equation*}
and
\begin{equation*}
B\bar{u}_{\ell'} \weaklyto b^\infty\,\, \textrm{ in } \mathcal{L}^2({L_2(U,H)}).
\end{equation*}
Finally, there is $\xi \in L^{q_0}(\Omega; H)$ such that $\bar{u}_{\ell'}(T) = u_{\ell'}(T) \weaklyto \xi$ in $L^{q_0}(\Omega; H)$.
\end{lemma}

\begin{proof}
The sequence $(\bar{u}_\ell)$ is bounded in $\mathcal{L}^2(V_1)$ due to Theorem~\ref{thm:apriori2} and in $\mathcal{L}^p(V_2)$ due to Corollary~\ref{corollary:remaining_apriori}.
The sequences $(A_1 \bar{u}_\ell)$, $(A_{2,\ell} \bar{u}_\ell)$ and $(B\bar{u}_\ell)$ are bounded in $\mathcal{L}^2(V_1^*)$, $\mathcal{L}^{p^*}(V_2^*)$ and in
$\mathcal{L}^{2}(L_2(U,H))$ respectively, due to Corollary~\ref{corollary:remaining_apriori}.
Finally the sequence $(u_\ell(T))$ is bounded in $L^{q_0}(\Omega; H)$ due to Theorem~\ref{thm:apriori2}.

Since it is assumed that $V_1$, $V_2$ are reflexive, it follows that $\mathcal{L}^2(V_1)$ and $\mathcal{L}^p(V_2)$ are reflexive.
A bounded sequence in a reflexive Banach space must have a weakly convergent subsequence (see e.g. Br\'ezis~\cite[Theorem 3.18]{brezis:functional}).
Applying this to the sequences in question concludes the proof of the lemma.
\end{proof}

Let $a^\infty := a_1^\infty + a_2^\infty$.
Then $a^\infty \in \mathcal{L}^{p^*}(V^*)$.
Due to Lemma~\ref{lemma:weak_limits_from_compactness} $a_\ell(\bar{u}_\ell) \weaklyto a^\infty$ in $\mathcal{L}^{p^*}(V^*)$
as $\ell' \to \infty$, provided that $\tfrac{\mathfrak{c}(m_\ell)}{n_\ell} \to 0$.
The following lemma provides the equation satisfied by the weak limits of the approximations.

\begin{lemma}
\label{lemma:limit_equation}
Let the Growth and Coercivity conditions in Assumption~\ref{ass:operators} hold.
Let Assumptions~\ref{ass:projection} and~\ref{ass:interpolation} be satisfied.
Let $u_\ell(0) \to u_0$ in $L^{q_0}(\Omega; H)$.
Let $\tfrac{\mathfrak{c}(m_\ell)}{n_\ell} \to 0$ as $\ell \to \infty$.
Then there is an $H$-valued adapted continuous process $\tilde{u}$ on $[0,T]$ 
such that $u=\tilde{u}$ $dt\times \P$-almost everywhere on $(0,T)\times \Omega$.
Furthermore, for almost every $(t,\omega) \in (0,T) \times \Omega$,
\begin{equation}
\label{eq:limit_equation1}
\tilde{u}(t) = u_0 + \int_0^t a^\infty (s) ds + \int_0^t b^\infty(s) dW(s)
\end{equation}
and almost surely
\begin{equation}
\label{eq:limit_equation2}
\tilde{u}(T) = u_0 + \int_0^T a^\infty (s) ds + \int_0^T b^\infty(s) dW(s).
\end{equation}
\end{lemma}
In the rest of this paper we will write $u$ instead of $\tilde{u}$ for notational simplicity.
\begin{proof}
Fix $M\in \N$.
Let $\varphi$ be a $V_M$-valued adapted stochastic process such that $|\varphi(t)| < M$ for all $t\in [0,T]$ and $\omega \in \Omega$.
For $g\in U$ let $\tilde{\Pi}_m g:= \sum_{j=1}^m (\chi_j,g)_U \chi_j$ and note that for any $v\in V$ one has $Bv\tilde{\Pi}_{m} \in L_2(U,H)$.
From~\eqref{eq:disc1} one observes that
\begin{equation}
\label{eq:limit_eq_proof_1}
\begin{split}
(u_{\ell'}(t),\varphi(t)) = (u_{\ell'}(0), & \varphi(t))
+ \bigg\langle \int_0^t a_{\ell'}(\bar{u}_{\ell'}(s)) ds, \varphi(t) \bigg\rangle\\
& + \bigg(\int_0^t B \bar{u}_{\ell'}(s)\tilde{\Pi}_{m_{\ell'}}dW(s) ,\varphi(t)\bigg).
\end{split}
\end{equation}
Let $G:\mathcal{L}^{p^*}(V^*) \to \mathcal{L}^{p^*}(V^*)$ be given by $(Gv)(t) := \int_0^t v(s)ds$.
Moreover, let $H:\mathcal{L}^2(L_2(U,H)) \to \mathcal{L}^2(H)$ be given by $(Hv)(t) := \int_0^t v(s) dW(s)$.
Integrating~\eqref{eq:limit_eq_proof_1} from $0$ to $T$ and taking the expectation yields
\begin{equation*}
\begin{split}
& \E\int_0^T (u_{\ell'}(t), \varphi(t)) dt =  \E\int_0^T (u_{\ell'}(0),\varphi(t))dt
\\ & + \E\int_0^T \big\langle (G a_{\ell'}(\bar{u}_{\ell'}))(t), \varphi(t) \big\rangle dt + \E \int_0^T \big( (H \, B \bar{u}_{\ell'}\tilde{\Pi}_{m_{\ell'}} )(t) ,\varphi(t)\big) dt.
\end{split}
\end{equation*}
The operator $G$ is linear and bounded and as such it is weakly-weakly continuous.
This operator $H$ is clearly linear.
Furthermore, due to It\^o's isometry,
\begin{equation*}
\begin{split}
\|Hv & \|_{\mathcal{L}^2(H)}^2  = \E\int_0^T |(Hv)(s)|^2 ds\\
& = \int_0^T \E \bigg|\int_0^t v(s) dW(s)\bigg|^2 dt = \E\int_0^T \int_0^t |v(s)|^2 ds dt \leq T \|v\|_{\mathcal{L}^2(L_2(U,H))}^2.
\end{split}
\end{equation*}
Thus the operator $H$ is also bounded.
It follows that $H$ is also weakly-weakly continuous.
Therefore, taking the limit as $\ell' \to \infty$ and using Lemma~\ref{lemma:weak_limits_from_compactness}, one obtains
\begin{equation*}
\begin{split}
& \E\int_0^T (u(t), \varphi(t)) dt =  \E\int_0^T (u_0,\varphi(t))dt
\\ & + \E\int_0^T \big\langle (G a^\infty)(t), \varphi(t) \big\rangle dt + \E \int_0^T \big( (H b^\infty)(t) ,\varphi(t)\big) dt.
\end{split}
\end{equation*}
This holds for any $\varphi$ as specified at the beginning of the proof.
By letting $M\to \infty$ and using the limited completeness of the Galerkin scheme it follows that this also holds for any $\varphi \in \mathcal{L}^{p}(V)$.
Thus
\begin{equation}
\label{eq:limit_eq_proof1a}
u(t) = u_0 + \int_0^t a^\infty (s) ds + \int_0^t b^\infty(s) dW(s)
\end{equation}
holds for almost all $(t,\omega) \in (0,T) \times \Omega$.

Let $\varphi$ be a $V_M$-valued and $\F_T$-measurable random variable such that $\E\|\varphi\|_V^2 < \infty$.
Setting $t=T$ in~\eqref{eq:limit_eq_proof_1} and taking the expectation yields
\begin{equation*}
\begin{split}
\E (u_{\ell'}(T), \varphi) = &  \E (u_{\ell'}(0),\varphi)
\\ & + \E \big\langle (G a_{\ell'}(\bar{u}_{\ell'}))(T), \varphi \big\rangle + \E \big( (H \,B\bar{u}_{\ell'}\tilde{\Pi}_{m'} )(T) ,\varphi\big).
\end{split}
\end{equation*}
Let $\ell' \to \infty$.
The weak-weak continuity of the operators $G$ and $H$,
together with Lemma~\ref{lemma:weak_limits_from_compactness}, 
implies that
\begin{equation}
\label{eq:limit_eq_proof_2}
\E (\xi, \varphi) =  \E (u_0,\varphi)
+ \E \big\langle (G a^\infty)(T), \varphi \big\rangle + \E \big( (H b^\infty)(T) ,\varphi\big).
\end{equation}
Letting $M\to \infty$ and again using the limited completeness of the Galerkin scheme shows that the above equality holds for any $\F_T$-measurable $\varphi \in L^2(\Omega;V)$.
If one now applies It\^o's formula to~\eqref{eq:limit_eq_proof1a} then one obtains an adapted process $\tilde{u}$ with paths in $C([0,T];H)$ that is equal to $u$ almost surely.
Furthermore, for any $\varphi \in L^2(\Omega; V)$ and due to continuity of $\tilde{u}$,
\begin{equation*}
\begin{split}
\E(\xi - \tilde{u}(T),\varphi) & = \lim_{t\to T}\E(\xi - \tilde{u}(t),\varphi) \\
& = \lim_{t\to T} \E \bigg\langle \int_t^T a^\infty(s) ds + \int_t^T b^\infty(s)dW(s),\varphi \bigg\rangle = 0.
\end{split}
\end{equation*}
Thus $\xi = \tilde{u}(T)$.
This together with~\eqref{eq:limit_eq_proof_2} implies~\eqref{eq:limit_equation2}.
\end{proof}

All that remains to be done to prove Theorem~\ref{thm:main} is to identify $a^\infty$ with $Au$ and $b^\infty$ with $Bu$ and to show strong convergence of $u_\ell(T)$ to $u(T)$.
To that end we would like to use monotonicity of $A$.
In order to overcome the difficulty arising from the fact that the tamed operator $A_{2,\ell}$ does not preserve the monotonicity property of $A_2$ we need the following lemma.
\begin{lemma}
\label{lemma:strong_conv_of_tam_correction}
Let the Growth and Coercivity conditions in Assumption~\ref{ass:operators} hold.
Let Assumptions~\ref{ass:projection} and~\ref{ass:interpolation} be satisfied.
Let $u_\ell(0) \to u_0$ in $L^{q_0}(\Omega; H)$.
Let $\tfrac{\mathfrak{c}(m_\ell)}{n_\ell} \to 0$ as $\ell \to \infty$.
Then
\begin{equation*}
\E\int_0^T \|A_2\bar{u}_\ell(s)-A_{2,\ell} \bar{u}_\ell(s)\|_{V_2^*}^{p^*} ds \to 0 \,\, \textrm{ as }\,\, \ell \to \infty.
\end{equation*}
\end{lemma}

\begin{proof}
Consider some $M>0$.
Recall that $T_\ell$ is given by~\eqref{eq:def_of_T_ell}.
Then
\begin{equation*}
\begin{split}
& I_\ell := \E\int_0^T \|A_{2,\ell} \bar{u}_\ell(s) - A_2\bar{u}_\ell(s)\|_{V_2^*}^{p^*} ds \\
& = \E\int_0^T \big(1-T_\ell(\bar{u}_\ell(s))\big)^{p^*}\|A_2 \bar{u}_\ell(s)\|_{V_2^*}^{p^*}\one{1}_{\{\|A_2 \bar{u}_\ell(s)\|_{V_2^*} \leq M\}} ds\\
& \quad+ \E\int_0^T \big(1-T_\ell(\bar{u}_\ell(s))\big)^{p^*}\|A_2 \bar{u}_\ell(s)\|_{V_2^*}^{p^*}\one{1}_{\{\|A_2 \bar{u}_\ell(s)\|_{V_2^*} > M\}} ds\\
& =: I_{1,\ell,M} + I_{2,\ell,M}.
\end{split}
\end{equation*}
It is observed that
\begin{equation*}
\begin{split}
I_{1,\ell,M} & \leq \E\int_0^T \frac{n_\ell^{-1/2}|\Pi_{m_\ell} A_2 \bar{u}_\ell(s)|}{1+n_\ell^{-1/2}|\Pi_{m_\ell} A_2 \bar{u}_\ell(s)|} \|A_2 \bar{u}_\ell(s)\|_{V_2^*}^{p^*}\one{1}_{\{\|A_2 \bar{u}_\ell(s)\|_{V_2^*} \leq M\}} ds\\
& \leq \E\int_0^T \frac{\tau_{n_\ell}^{1/2}T^{-1/2} \mathfrak{c}(m)^{1/2}M}{1+n_\ell^{-1/2}|\Pi_{m_\ell} A_2 \bar{u}_\ell(s)|} M^{p^*} ds
\leq (\mathfrak{c}(m)\tau_{n_\ell} )^{1/2} T^{1/2} M^{1+p^*}.
\end{split}
\end{equation*}
Recall that due to Corollary~\ref{corollary:remaining_apriori} one knows that
\begin{equation*}
\E\int_0^T \|A_2 \bar{u}_\ell(s)\|_{V_2^*}^{p^*(1+\eta)} ds < c
\end{equation*}
with $c$ independent of $\ell$.
Thus the sequence $\big(\|A_2 \bar{u}_\ell\|_{V_2^*}^{p^*}\big)_{\ell \in \N}$ is uniformly integrable on $(0,T)\times \Omega$ with respect to $dt \times P$.
Hence for any $\epsilon > 0$ there exists $M$ such that $I_{2,\ell,M} < \epsilon/2$ for all $\ell$.
Finally, since $\tfrac{\mathfrak{c}(m_\ell)}{n_\ell} \to 0$ as $\ell \to \infty$, one can choose $\ell$ large such that $I_{1,\ell,M} < \epsilon/2$.
\end{proof}

We now prove Theorem~\ref{thm:uniqueness}. 
This is needed to later show that the whole sequence of approximations
converges rather than just a subsequence.

\begin{proof}[Proof of Theorem~\ref{thm:uniqueness}]
Assume that $u_1$ and $u_2$ are two distinct solutions to~\eqref{eq:see}
such that $u_1(0) = u_2(0) = u_0$.
One would now like to apply It\^o's formula for the square of the norm from Pardoux~\cite[Chapitre 2, Theoreme 5.2]{pardoux:thesis}.
To that end one immediately observes that $u_1 - u_2 \in \mathcal{L}^2(V_1) \cap \mathcal{L}^p(V_2)$ and that $u_1(0) - u_2(0) = 0 \in L^2(\Omega;H)$.
Moreover
\begin{equation*}
\|Au_1 - Au_2\|_{\mathcal{L}^2(V_1^*) + \mathcal{L}^{p^*}(V_2^*)} = \|A_1 u_1 - A_1 u_2\|_{\mathcal{L}^2(V_1^*)} + \|A_2 u_1 - A_2 u_2\|_{\mathcal{L}^{p^*}(V_2^*)}.
\end{equation*}
Using the Growth assumption on $A_1$ one observes that
\begin{equation*}
\begin{split}
& \E \int_0^T \|A_1 u_1(s) - A_1 u_2(s)\|_{V_1^*}^2 ds \leq \E \int_0^T 2\big[\|A_1 u_1(s)\|_{V_1^*}^2 + \|A_1 u_2(s)\|_{V_1^*}^2 \big] ds\\
& \leq K \E \int_0^T 2\big[(1+\|u_1(s)\|_{V_1}^2) + (1+\|u_2(s)\|_{V_1}^2) \big] ds\\
& \leq c(1 + \|u_1\|_{\mathcal{L}^2(V_1)}^2 + \|u_2\|_{\mathcal{L}^2(V_1)}^2) < \infty.
\end{split}
\end{equation*}
Also, using the Growth assumption on $A_2$ one obtains
\begin{equation*}
\begin{split}
\E \int_0^T \|A_2 u_1(s) - A_2 u_2(s)\|_{V_2^*}^{p^*} ds
\leq c(1 + \|u_1\|_{\mathcal{L}^p(V_2)}^p + \|u_2\|_{\mathcal{L}^p(V_2)}^p) < \infty.
\end{split}
\end{equation*}
Thus $Au_1  - Au_2 \in \mathcal{L}^2(V_1^*)\cup \mathcal{L}^{p^*}(V_2^*)$.
Finally, using the Growth assumption on $B$ one deduces that $Bu_1 - Bu_2 \in \mathcal{L}^2(L_2(U,H))$.
Hence the afromentioned It\^o's formula for the square of the norm can be applied, yielding
\begin{equation*}
\begin{split}
e^{-Kt}|u_1(t) - u_2(t)|^2  =
& -K \int_0^t e^{-Ks} |u_1(s) - u_2(s)|^2 ds\\
& + \int_0^t e^{-Ks}\big[ 2\langle Au_1(s) - Au_2(s),u_1(s)-u_2(s)\rangle \\
 & + \|Bu_1(s) - Bu_2(s)\|_{L_2(U,H)}^2 \big] ds + M(t),
\end{split}
\end{equation*}
where
\begin{equation*}
M(t) := \int_0^t e^{-Ks}\big(u_1(s)-u_2(s),(Bu_1(s)-Bu_2(s))dW(s)\big).
\end{equation*}
One then observes, due to the monotonicity of $A:V\times \Omega \to V^*$, that
\begin{equation*}
e^{-Kt}|u_1(t)-u_2(t)|^2 \leq M(t)
\end{equation*}
and hence $M$ is non-negative.
It is also a real-valued continuous local martingale, and thus a supermartingale.
Furthermore it starts from $0$ and thus, almost surely, $M(t) = 0 $ for all $t\in [0,T]$.
One thus concludes that
$u_1(t) = u_2(t)$ for all $t\in [0,T]$ almost surely.
\end{proof}

Finally we can prove Theorem~\ref{thm:main}.

\begin{proof}[Proof of Theorem~\ref{thm:main}]
Recall that $a_\ell(v) := \Pi_{m_\ell} [A_1 v + A_{2,\ell} v]$.
Applying It\^o's formula to the scheme~\eqref{eq:disc1} and taking expectations yields
\begin{equation*}
\begin{split}
e^{-KT} & \E |u_{\ell'}(T)|^2  =  \E|u_{\ell'}(0)|^2 -K \E \int_0^T e^{-Ks} |u_{\ell'}(s)|^2 ds \\
& + \E \int_0^T e^{-Ks}\Big( 2 \langle a_{\ell'}(\bar{u}_{\ell'}(s)), \bar{u}_{\ell'}(s) \rangle + \|\Pi_{m_{\ell'}} B \bar{u}_{\ell'}(s) \|_{L_2(U,H)}^2 \Big) ds\\
& + \E \int_0^T 2e^{-Ks} \langle a_{\ell'}(\bar{u}_{\ell'}(s)), u_{\ell'}(s) - \bar{u}_{\ell'}(s) \rangle ds.
\end{split}
\end{equation*}
Let
\begin{equation*}
I_{1,\ell'} := \E \int_0^T \langle a_{\ell'}(\bar{u}_{\ell'}(s)), u_{\ell'}(s) - \bar{u}_{\ell'}(s) \rangle ds.
\end{equation*}
Using H\"older's inequality results in
\begin{equation*}
I_{1,\ell'} \leq \bigg(\E\int_0^T |a_\ell( \bar{u}_{\ell'}(s))|^2 ds \bigg)^{1/2}\bigg(\E\int_0^T |u_{\ell'}(s) - \bar{u}_{\ell'}(s)|^2 ds \bigg)^{1/2}.
\end{equation*}
Using Assumption~\ref{ass:projection} and Corollary~\ref{corollary:remaining_apriori} yields
\begin{equation*}
I_{1,\ell'} \leq c (\mathfrak{c}(m_{\ell'})\tau_{n_{\ell'}})^{1/2}.
\end{equation*}
Thus,
\begin{equation*}
\begin{split}
& e^{-Ks}\E |u_{\ell'}(T)|^2 \leq \E|u_{\ell'}(0)|^2 -K \E \int_0^T e^{-Ks}|u_{\ell'}(s)|^2 ds\\
& + \E \int_0^T e^{-Ks}\big( 2 \langle a_{\ell'}(\bar{u}_{\ell'}(s)), \bar{u}_{\ell'}(s) \rangle + \|B \bar{u}_{\ell'}(s)\|_{L_2(U,H)}^2 \big) ds
+ c(\mathfrak{c}(m_{\ell'})\tau_{n_{\ell'}})^{1/2}
\end{split}
\end{equation*}
and one may proceed with a monotonicity argument.
Let $w\in \mathcal{L}^p(V)$. Then
\begin{equation*}
\begin{split}
&e^{-KT}\E|u_{\ell'}(T)|^2 \leq \E|u_{\ell'}(0)|^2-K \E\int_0^T e^{-Ks}|u_{\ell'}(s)|^2 ds\\
&\quad + \E \int_0^T e^{-Ks}\big[ 2 \langle a_{\ell'}(\bar{u}_{\ell'}(s)) - a_{\ell'}(w(s)), \bar{u}_{\ell'}(s) - w(s)\rangle \\
& \quad \quad + 2 \langle a_{\ell'}(w(s)), \bar{u}_{\ell'}(s)) - w(s)\rangle + 2\langle a_{\ell'}(\bar{u}_{\ell'}(s)), w(s) \rangle \\
& \quad \quad + 2(Bw(s),B\bar{u}_{\ell'}(s))_{L_2(U,H)\times L_2(U,H)} - \|Bw(s)\|_{L_2(U,H)}^2\\
& \quad \quad + \|B\bar{u}_{\ell'}(s)) - Bw(s)\|_{L_2(U,H)}^2 \big] ds
+  c(\mathfrak{c}(m_{\ell'})\tau_{n_{\ell'}})^{1/2}.
\end{split}
\end{equation*}
Using the Monotonicity assumption on $A$ one obtains
\begin{equation}
\label{eq:main_thm_proof1}
\begin{split}
&e^{-KT}\E|u_{\ell'}(T)|^2 \leq \E|u_{\ell'}(0)|^2\\
& + K\E\int_0^T e^{-Ks} \big(|\bar{u}_{\ell'}(s)|^2 -|u_{\ell'}(s)|^2 -2(\bar{u}_{\ell'}(s),w(s)) + |w(s)|^2 \big) ds \\
&+ \E \int_0^T  e^{-Ks} \langle A_{2,\ell'}\bar{u}_{\ell'}(s) - A_2 \bar{u}_{\ell'}(s), \bar{u}_{\ell'}(s) - w(s)\rangle ds\\
&  + \E \int_0^T  e^{-Ks} \langle A_2 w(s) - A_{2,\ell'}w(s), \bar{u}_{\ell'}(s) - w(s)\rangle ds\\
& + \E \int_0^T  e^{-Ks} \big[ 2 \langle a_{\ell'}(w(s)), \bar{u}_{\ell'}(s)) - w(s)\rangle\\
& \quad\quad + 2\langle a_{\ell'}(\bar{u}_{\ell'}(s)), w(s) \rangle \\
& \quad\quad + 2(Bw(s),B\bar{u}_\ell(s))_{L_2(U,H)\times L_2(U,H)} - \|Bw(s)\|_{L_2(U,H)}^2 \big] ds\\
& + c(\mathfrak{c}(m_{\ell'})\tau_{n_{\ell'}})^{1/2}.
\end{split}
\end{equation}
Taking limit inferior as $\ell' \to \infty$, 
using the weak lower-semi-continuity of the norm,
Lemma~\ref{lemma:weak_limits_from_compactness},
Lemma~\ref{lemma:strong_conv_of_tam_correction} 
and Corollary~\ref{corollary:remaining_apriori},
one observes that
\begin{equation}
\label{eq:main_thm_proof2}
\begin{split}
& e^{-KT}\E|u(T)|^2 \leq \E|u_0|^2 + K\E\int_0^T e^{-Ks} \big[ -2(u(s),w(s)) + |w(s)|^2 \big] ds \\
& + \E\int_0^T e^{-Ks}\big[ 2 \langle A w(s), u(s) - w(s) \rangle + 2\langle a^\infty(s), w(s)  \rangle  \\
&  \quad\quad\quad + 2(Bw(s), b^\infty(s))_{L_2(U,H)\times L_2(U,H)} - \|Bw(s)\|_{L_2(U,H)}^2 \big] ds.
\end{split}
\end{equation}
Applying It\^o's formula to~\eqref{eq:limit_equation1} and taking expectations yields
\begin{equation}
\label{eq:main_thm_proof3}
\begin{split}
e^{-KT}\E|u(T)|^2 = & \E|u(0)|^2 - K\E \int_0^T e^{-Ks} |u(s)|^2 ds\\
& + \E\int_0^T e^{-Ks} \big[2\langle a^\infty(s),u(s)\rangle + \|b^\infty(s)\|_{L_2(U,H)}^2 \big]ds
\end{split}
\end{equation}
Subtracting this from~\eqref{eq:main_thm_proof2} one arrives at
\begin{equation}
\label{eq:main_thm_proof4}
\begin{split}
0 \leq  \E\int_0^T &  e^{-Ks}\big[  K|u(s)-w(s)|^2 + 2 \langle A w(s), u(s) - w(s) \rangle  \\
& + 2 \langle a^\infty(s), w(s) - u(s)\rangle - \|Bu(s) - b^\infty(s)\|_{L_2(U,H)}^2 \big]ds.
\end{split}
\end{equation}
Note that so far $w$ was arbitrary.
It will now be used  to identify the nonlinear terms.
First, one takes $w=u$ and observes that,
\begin{equation*}
0 \leq - \E\int_0^T  e^{-Ks} \|Bu(s) - b^\infty(s)\|_{L_2(U,H)}^2 ds \leq 0
\end{equation*}
which implies $b^\infty = Bu$.
Next, one sets $w=u+\epsilon z$ with $\epsilon > 0$ 
and $z\in \mathcal{L}^p(V)$ in~\eqref{eq:main_thm_proof4}.
Dividing by $\epsilon > 0$ leads to
\begin{equation*}
0 \leq \E\int_0^T e^{-Ks} \big[ K\epsilon|z|^2+ 2 \langle A (u(s)+\epsilon z(s)), -z(s) \rangle + 2 \langle a^\infty(s), z(s)\rangle \big]ds.
\end{equation*}
Using hemicontinuity of $A$ while letting $\epsilon \to 0$ results in
\begin{equation*}
\E\int_0^T \langle A u(s), z(s) \rangle ds \leq \E\int_0^T \langle a^\infty(s), z(s)\rangle ds.
\end{equation*}
This holds for an arbitrary $z\in \mathcal{L}^p(V)$ and hence, 
in particular, for $-z$. 
Thus one obtains that $a^\infty = Au$.

Due to Theorem~\ref{thm:uniqueness}, the solution $u$ to~\eqref{eq:see} is unique.
Thus the whole sequences of approximations converges rather than just the subsequence denoted by $\ell'$.

Finally, in order to show that $u_\ell(T) \to u(T)$ in $L^2(\Omega;H)$,
one uses~\eqref{eq:main_thm_proof1} and~\eqref{eq:main_thm_proof2}
with $w=u$ together with $a^\infty = Au$ and $b^\infty = Bu$.
Consequently, the weak-lower-semi-continuity of the norm and
Lemma~\ref{lemma:weak_limits_from_compactness} lead to
\begin{equation*}
\begin{split}
e^{-KT}\E|u(T)|^2  \leq & \liminf_{\ell\to \infty} e^{-KT} \E |u_\ell(T)|^2 \leq \E|u_0|^2 -K \E \int_0^T e^{-Ks} |u(s)|^2 ds \\
&  + \E \int_0^T e^{-Ks}\big[2\langle Au(s),u(s) \rangle + \|Bu(s)\|_{L_2(U,H)}^2 \big]ds.
\end{split}
\end{equation*}
Thus, due to~\eqref{eq:main_thm_proof3},
\begin{equation*}
0 \leq \liminf_{\ell\to \infty} \E |u_\ell(T)|^2 - \E|u(T)|^2 \leq 0.
\end{equation*}
From Lemma~\ref{lemma:weak_limits_from_compactness} and Lemma~\ref{lemma:limit_equation}, one already knows that $u_\ell(T) \weaklyto u(T)$ in $L^2(\Omega;H)$.
This is a uniformly convex space (as it is a Hilbert space). 
Thus one concludes that $u_\ell(T) \to u(T)$ in $L^2(\Omega;H)$.
For this see, e.g., Br\'ezis~\cite[Proposition 3.32]{brezis:functional}.
\end{proof}

\section{Examples}
\label{sec:examples}
In this section we give examples of three equations which fit into our 
framework.
In all three examples the interpolation inequality is a consequence of the Gagliardo--Nirenberg inequality (see, for example,~\cite[Theorem 1.24]{roubicek:nonlinear}).
The first example is the  equation:
\begin{equation*}
du = \left[\nabla a(\nabla u)  - |u|^{p-2} u\right]dt + u dW\,\, \textrm{ on } \,\, \mathscr{D}\times (0,T)
\end{equation*}
with $u=0$ on the boundary of the domain $\mathscr{D}$ 
and $u(\cdot,0) = u_0$ given.
Here $a:\R^d \to \R^d$ can be nonlinear 
but it is assumed to be continuous, monotone and growing at most linearly.
If we take $a_i(z) = z_i$ then $\nabla a(\nabla u) = \Delta u$ and this equation
is the stochastic Ginzburg--Landau equation.
An example of a nonlinear function is 
$a_i(z) = \tfrac{2+\exp(-z_i)}{1+\exp(-z_i)}$. 
Moreover $\mathscr{D}$ is a bounded Lipschitz domain in $\R^d$, $d=1,2,3$ and $p\in [2,6)$ if $d=1$, $p\in [2,4)$ if $d=2$ and $p\in [2,10/3)$ if $d=3$.
In our framework $H=L^2(\mathscr{D})$, $V_1 = H^1_0(\mathscr{D})$ and $V_2 = L^p(\mathscr{D})$ (using the standard notation for Lebesgue and Sobolev spaces).

The second is the stochastic Swift--Hohenberg equation:
\begin{equation*}
du = \big[\left(\gamma^2 - (1+\Delta)^2\right)u - |u|^{p-2}u\big]dt + dW\,\, \textrm{ on } \,\, \mathscr{D}\times (0,T)
\end{equation*}
with appropriate boundary and initial conditions.
The domain $\mathscr{D}$ is assumed to be a bounded Lipschitz domain in $\R^2$.
With Dirichlet boundary conditions we would take $V_1 = H^2_0(\mathscr{D})$ 
and $V_2 = L^p(\mathscr{D})$ with $p\in [2,6)$.

The third example is the spatially extended stochastic FitzHugh--Nagumo system for signal propagation in nerve cells (originally stated by FitzHugh~\cite{fitzhugh:impulses} as a system of ordinary differential equations, see Bonaccorsi and Mastrogiacomo~\cite{bonaccorsi:analysis} for mathematical analysis of the spatially extended stochastic version):
\begin{equation*}
\begin{array}{ll}
du & = (\Delta u + u - u^3 - v)dt + dW\\
dv & = c_1(u - c_2 v + c_3)dt
\end{array}
\,\, \textrm{ on } \,\, (0,1)\times (0,T),
\end{equation*}
together with appropriate initial data for $u$ and $v$ as well as homogeneous Neumann boundary conditions for $u$ only.
In this situation $V_1 = H^1((0,1))\times L^2((0,1))$ while $V_2 = L^4((0,1))\times L^2((0,1))$.

We now provide estimates on the constant $\mathfrak{c}(m)$ in the particular 
case when $\mathscr{D} = (0,\pi)^2 \subset \R^2$ and
we use a spectral Galerkin method to construct the spaces $V_m$.
To that end define 
\begin{equation*}
\varphi_{n_1 n_2}(x_1,x_2) := \frac{2}{\pi}\sin(n_1 x_1)\sin(n_2 x_2). 	
\end{equation*}
Let $V_m = \textrm{span}\{\varphi_{n_1 n_2} : n_1 = 1,\ldots, m, 
n_2 = 1,\ldots,m\}$.
Then 
\begin{equation*}
\Pi_m f := 
\sum_{n_1=1, n_2=1}^m \langle f, \varphi_{n_1 n_2} \rangle \varphi_{n_1 n_2}
\end{equation*}
satisfies Assumption~\ref{ass:projection}.
Moreover we can calculate 
\begin{equation*}
\begin{split}
\mathfrak{c}(m) & = \sum_{n_1=1, n_2=1}^m 
\left(
\|\varphi_{n_1 n_2}\|_{L^2(\mathscr{D})}^2 
+ \|\nabla \varphi_{n_1 n_2}\|_{L^2(\mathscr{D};\R^2)}^2
+ \|\varphi_{n_1 n_2}\|_{L^p(\mathscr{D})}^{2/p} 
\right)	\\
& = m^2 \left(1 + 2 + c_p\right),	
\end{split}
\end{equation*}
where $c_p$ depends only on $p$.
Hence, in order to apply Theorem~\ref{thm:main}, we need a sequence 
$(m_\ell, n_\ell, k_\ell)$ such that $\tfrac{m_\ell^2}{n_\ell} \to 0$ 
as $\ell \to \infty$. 
This means that we need to choose $n_\ell = \lfloor m_\ell^{2+\delta} \rfloor$
for some $\delta > 0$. 

We also note that if $\mathscr{D} = (0,\pi)^d$ then an analogous construction
of $V_m$ would lead to the conclusion that we need 
$n_\ell = \lfloor m_l^{d+\delta}\rfloor$ for some $\delta > 0$.
Crucially we see that the space-time coupling requirement is no more onerous
than in the case of equations with operators growing at most linearly.

\section*{Acknowledgement}
The authors would like to thank the referees for their comments on the paper.


\begin{thebibliography}{10}


\bibitem{bonaccorsi:analysis}
S.~Bonaccorsi and E.~Mastrogiacomo.
\newblock Analysis of the Stochastic FitzHugh--Nagumo System. 
\newblock {\em Infin. Dimens. Anal. Quantum Probab. Relate. Top.}, 11(03):427--446, 2008.


\bibitem{brezis:functional}
H.~Br\'ezis.
\newblock {\em {Functional analysis, {S}obolev spaces and partial differential
  equations}}.
\newblock Springer, New York, 2010.

\bibitem{fitzhugh:impulses}
R.~FitzHugh.
\newblock Impulses and Physiological States in Theoretical Models of Nerve Membrane.
\newblock {\em Biophys. J.}, 1:445--466, 1961.

\bibitem{ggz}
H.~Gajewski, K.~Gr{\"o}ger, and K.~Zacharias.
\newblock {\em Nichtlineare {O}peratorgleichungen und
  {O}peratordifferentialgleichungen}.
\newblock Akademie-Verlag, Berlin, 1974.

\bibitem{giles:multilevel}
M.B.~Giles.
\newblock{Multilevel Monte Carlo path simulation}.
\newblock{\em Oper. Res.},  56:607-–617, 2008.

\bibitem{gyongy:millet:on:discretization}
I.~Gy{\"o}ngy and A.~Millet.
\newblock On discretization schemes for stochastic evolution equations.
\newblock {\em Potential Anal.}, 23(2):99--134, 2005.

\bibitem{heinrich:monte}
S.~Heinrich.
\newblock Monte Carlo complexity of global solution of integral equations.
\newblock {\em J. Complexity}, 14(2):151--175, 1998.

\bibitem{heinrich:multilevel}
S.~Heinrich.
\newblock Multilevel Monte Carlo methods. 
In {\em Large-Scale Scientific Computing}, vol. 2179, pp. 58--67 of
\newblock {\em Lect. Notes Comput. Sci.}, 2001.	

\bibitem{HJK-div} M.~Hutzenthaler and A.~Jentzen and P.E.~Kloeden.
\newblock Strong and weak divergence in finite time of Euler's method for
stochastic differential equations with non-globally Lipschitz
continuous coefficients. 
\newblock \emph{Proc. R. Soc. A}  467:1563--1576, 2011. 

\bibitem{hutzenthaler:jentzen:kloeden:strong:convergence}
M.~Hutzenthaler and A.~Jentzen and P.E.~Kloeden.
\newblock Strong convergence of an explicit numerical method for SDEs with non-globally Lipschitz continuous coefficients.
\newblock {\em Ann. Appl. Probab.}, 22:1611--1641, 2012.

\bibitem{hutzenthaler:jentzen:kloeden:strong:divergence:multilevel}
M.~Hutzenthaler and A.~Jentzen and P.E.~Kloeden.
\newblock Divergence of the multilevel Monte Carlo Euler method for nonlinear stochastic differential equations.
\newblock \emph{Ann. Appl. Probab.}, 23(5):1913--1966, 2013. 

\bibitem{hutzenthaler:jentzen:numerical}
M.~Hutzenthaler and A.~Jentzen. 
\newblock Numerical Approximations of stochastic differential equations
with non-globally Lipschitz coefficients. 
\newblock{\em Mem. Amer. Math. Soc.} 4:1--112, 2015.


%

\bibitem{pardoux:thesis}
{\'E}.~Pardoux.
\newblock Sur des \'equations aux d\'eriv\'ees partielles stochastiques
  monotones.
\newblock {\em C. R. Acad. Sci. Paris S\'er. A-B}, 275:A101--A103, 1972.

\bibitem{roberts:tweedie:exponential}
G. O.~Roberts and R. L.~Tweedie. 
\newblock Exponential convergence of Langevin distributions and their 
discrete approximations.
\newblock {\em Bernoulli}, 2(4):341--363.


\bibitem{roubicek:nonlinear}
T.~Roub{\'{\i}}{\v{c}}ek.
\newblock {\em Nonlinear partial differential equations with applications}.
\newblock Birkh\"auser, Basel, 2005.

\bibitem{sabanis:note}
S.~Sabanis.
\newblock A note on tamed Euler approximations.
\newblock {\em Electron. Commun. Probab.}, 18(47):1--10, 2013.

\bibitem{sabanis:euler} 
S. Sabanis.
\newblock Euler approximations with varying coefficients: 
the case of superlinearly growing diffusion coefficients. 
\newblock {\em arXiv}:1308.1796 [math.PR], 2015. 

\bibitem{dareitos:kumar:sabanis} 
K. Dareiotis, C. Kumar and S. Sabanis.
\newblock On Tamed Euler Approximations of SDEs 
Driven by L\'{e}vy Noise with Applications to Delay Equations. 
{\em arXiv}:1403.0498 [math.PR], 2015.

\bibitem{tretyakov:zhang:fundamental}
M. V.~Tretyakov and Z.~Zhang. 
\newblock A fundamental mean-square convergence theorem for {SDE}s with 
locally {L}ipschitz coefficients and its applications.
\newblock {\em SIAM J. Numer. Anal.}, 51(6):3135--3162, 2013.




\end{thebibliography}
\end{document}